\documentclass[a4paper,12pt]{article}
\usepackage[utf8]{inputenc}

\usepackage{mathrsfs}
\usepackage{graphicx}
\usepackage{hyperref}
\usepackage{enumerate}
\usepackage{multicol}
\usepackage{color}
\usepackage{amsmath,amssymb,amscd}
\usepackage{amsthm}

\usepackage{pdfpages}

\usepackage{times}
\usepackage[varg]{txfonts}

\usepackage[top=1in, bottom=1in, left=0.75in, right=0.75in]{geometry}

\theoremstyle{plain}
\newtheorem{thm}{Theorem}
\newtheorem{lem}{Lemma}
\newtheorem{cor}{Corollary}
\newtheorem{prop}{Proposition}

\theoremstyle{definition}
\newtheorem{dfn}{Definition}
\newtheorem{ex}{Example}

\theoremstyle{remark}
\newtheorem{rem}{Remark}
\newtheorem{opp}{Open Problem}
\newtheorem*{ackn}{Acknowledgment}

\title{Polynomial identities satisfied by generalized polynomials}
\author{Eszter Gselmann}

\begin{document}
\maketitle
 \begin{abstract}
 The main purpose of this paper is solve polynomial equations that are satisfied by 
 (generalized) polynomials. More exactly, we deal with the following problem: 
 let $\mathbb{F}$ be a field with $\mathrm{char}(\mathbb{F})=0$ and $P\in \mathbb{F}[x]$ and $Q\in \mathbb{C}[x]$ be polynomials. Our aim is to prove characterization theorems for generalized polynomials $f\colon \mathbb{F}\to \mathbb{C}$ of degree two that also fulfill equation 
\[
 f(P(x))= Q(f(x))
\]
for each $x\in \mathbb{F}$. 
As it turns out, the difficulty of such problems heavily depends on that we consider the above equation for generalized polynomials or for (normal) polynomials. Therefore, firstly we study the connection between these two notions. 
\end{abstract}

\begin{quotation}
\begin{center}
 This paper is dedicated to Professors \emph{György Gát} and \emph{Zsolt Páles} on the occasion of their $60$\textsuperscript{th} and $65$\textsuperscript{th} birthday, respectively. 
 \end{center}
\end{quotation}

\section{Introduction and preliminaries}

\subsection{Homomorphisms and derivations}

The study of additive mappings from a ring into another ring which
preserve squares was initiated by G.~Ancochea in \cite{Anc42} in
connection with problems arising in projective geometry. Later,
these results were strengthened by (among others) Kaplansky
\cite{Kap} and Jacobson--Rickart \cite{JR}.

 Let $R, R'$ be rings, the mapping $\varphi:R\rightarrow R'$ is called a \emph{homomorphism} if
\[
 \varphi(a+b)=\varphi(a)+\varphi(b)
\qquad \left(a, b\in R\right)
\]
and
\[
 \varphi(ab)=\varphi(a)\varphi(b)
\qquad \left(a, b\in R\right).
\]
Furthermore, the function $\varphi:R\to R'$ is an
\emph{anti-homomorphism} if
\[
 \varphi(a+b)=\varphi(a)+\varphi(b)
\qquad \left(a, b\in R\right)
\]
and
\[
 \varphi(ab)=\varphi(b)\varphi(a)
\qquad \left(a, b\in R\right).
\]

Henceforth, $\mathbb{N}$ will denote the set of the positive
integers. Let $n\in\mathbb{N}, n\geq 2$ be fixed. 
The function $\varphi:R\rightarrow R'$ is called an $n$-Jordan
homomorphism if
\[
\varphi(a+b)=\varphi(a)+\varphi(b) \qquad \left(a, b\in R\right)
\]
and
\[
 \varphi(a^{n})=\varphi(a)^{n}
\qquad \left(a\in R\right).
\]
Finally, we remark that in case $n=2$ we speak about homomorphisms
and Jordan homomorphisms, respectively. It was G.~Ancochea who
firstly dealt with the connection of Jordan homomorphisms and
homomorphisms, see \cite{Anc42}. These results were
generalized and extended in several ways, see for instance
\cite{JR}, \cite{Kap}, \cite{Zel68}.

Let $n\in\mathbb{N}$, we say that a ring $R$ is \emph{of
characteristic larger than $n$} if $n!x=0$ implies that $x=0$. The
ring $R$ is termed to be a \emph{prime ring} if
$ a, b\in R$ and $aRb=\left\{0\right\}$ imply that either $a=0$ or $b=0$. 
In \cite{Her} I.N.~Herstein proved that 
if $\varphi$ is a Jordan homomorphism of a ring $R$
\emph{onto} a prime ring $R'$ of characteristic different from $2$
and $3$, then either $\varphi$ is a homomorphism or an
anti-homomorphism.

Furthermore, in the above-mentioned paper \cite{Her}, 
not only Jordan homomorphisms but also $n$-Jordan
mappings were considered and he proved the following statement. 
\begin{thm}
Let $\varphi$ be an $n$-Jordan
homomorphism from a ring $R$ \textbf{onto} a prime ring $R'$ and assume that $R'$ has 
characteristic larger than $n$, suppose further that $R$ has a unit
element. Then $\varphi=\varepsilon\tau$ where $\tau$ is either a
homomorphism or an anti-homomorphism and $\varepsilon$  is an
$(n-1)$st root of unity lying in the center of $R'$.
\end{thm}

Besides homomorphisms, derivations also play a key role in the theory of rings and fields. Concerning this notion, we will follow \cite[Chapter 14]{Kuc}. 

Let $Q$ be a ring and let $P$ be a subring of $Q$.
A function $d\colon P\rightarrow Q$ is called a \emph{derivation}\index{derivation} if it is additive,
i.e. 
\[
d(x+y)=d(x)+d(y)
\quad
\left(x, y\in P\right)
\]
and also satisfies the so-called \emph{Leibniz rule}\index{Leibniz rule}, i.e.  equation
\[
d(xy)=d(x)y+xd(y)
\quad
\left(x, y\in P\right). 
\]

Fundamental examples for derivations are the following ones.

Let $\mathbb{F}$ be a field, and let in the above definition $P=Q=\mathbb{F}[x]$
be the ring of polynomials with coefficients from $\mathbb{F}$. For a polynomial
$p\in\mathbb{F}[x]$, $p(x)=\sum_{k=0}^{n}a_{k}x^{k}$, define the function
$f\colon \mathbb{F}[x]\rightarrow\mathbb{F}[x]$ as
\[
f(p)=p',
\]
where $p'(x)=\sum_{k=1}^{n}ka_{k}x^{k-1}$ is the derivative of the polynomial $p$.
Then the function $f$ clearly fulfills
\[
f(p+q)=f(p)+f(q)
\]
and
\[
f(pq)=pf(q)+qf(p)
\]
for all $p, q\in\mathbb{F}[x]$. Hence $f$ is a derivation.

 Let $(\mathbb{F}, +, \cdot)$ be a field, and suppose that we are given a derivation 
 $f\colon \mathbb{F}\to \mathbb{F}$. We define the mapping $f_{0}\colon \mathbb{F}[x]\to \mathbb{F}[x]$ in the following way. 
If $p\in \mathbb{F}[x]$ has the form 
\[
 p(x)=\sum_{k=0}^{n}a_{k}x^{k}, 
\]
then let 
\[
 f_{0}(p)= p^{f}(x)=\sum_{k=0}^{n}f(a_{k})x^{k}. 
\]
Then $f_{0}\colon \mathbb{F}[x]\to \mathbb{F}[x]$ is a derivation.

It is well-known that in case of additive functions, Hamel bases play an important role. 
As \cite[Theorem 14.2.1]{Kuc} shows in case of derivations, algebraic bases are fundamental. 

\begin{thm}\label{T14.2.1}
Let $(\mathbb{K}, +,\cdot)$ be a field of characteristic zero, let $(\mathbb{F}, +,\cdot)$
be a subfield of $(\mathbb{K}, +,\cdot)$, let $S$ be an algebraic base of $\mathbb{K}$ over $\mathbb{F}$,
if it exists, and let $S=\emptyset$ otherwise.
Let $f\colon \mathbb{F}\to \mathbb{K}$ be a derivation.
Then, for every function $u\colon S\to \mathbb{K}$,
there exists a unique derivation $g\colon \mathbb{K}\to \mathbb{K}$
such that $g \vert_{\mathbb{F}}=f$ and $g \vert_{S}=u$.
\end{thm}

\subsection{The symmetrization method}

While proving our results, the so-called Polarization formula for multi-additive functions and the symmetrization method will play a key role. In this subsection the most important notations and statements are summarized. Here we follow the monograph \cite{Sze91}.

\begin{dfn}
 Let $G, S$ be commutative semigroups, $n\in \mathbb{N}$ and let $A\colon G^{n}\to S$ be a function.
 We say that $A$ is \emph{$n$-additive} if it is a homomorphism of $G$ into $S$ in each variable.
 If $n=1$ or $n=2$ then the function $A$ is simply termed to be \emph{additive}
 or \emph{bi-additive}, respectively.
\end{dfn}

The \emph{diagonalization} or \emph{trace} of an $n$-additive
function $A\colon G^{n}\to S$ is defined as
 \[
  A^{\ast}(x)=A\left(x, \ldots, x\right)
  \qquad
  \left(x\in G\right).
 \]
As a direct consequence of the definition each $n$-additive function
$A\colon G^{n}\to S$ satisfies
\[
 A(x_{1}, \ldots, x_{i-1}, kx_{i}, x_{i+1}, \ldots, x_n)
 =
 kA(x_{1}, \ldots, x_{i-1}, x_{i}, x_{i+1}, \ldots, x_{n})
 \qquad 
 \left(x_{1}, \ldots, x_{n}\in G\right)
\]
for all $i=1, \ldots, n$, where $k\in \mathbb{N}$ is arbitrary. The
same identity holds for any $k\in \mathbb{Z}$ provided that $G$ and
$S$ are groups, and for $k\in \mathbb{Q}$, provided that $G$ and $S$
are linear spaces over the rationals. For the diagonalization of $A$
we have
\[
 A^{\ast}(kx)=k^{n}A^{\ast}(x)
 \qquad
 \left(x\in G\right).
\]

The above notion can also be extended for the case $n=0$ by letting 
$G^{0}=G$ and by calling $0$-additive any constant function from $G$ to $S$. 

One of the most important theoretical results concerning
multiadditive functions is the so-called \emph{Polarization
formula}, that briefly expresses that every $n$-additive symmetric
function is \emph{uniquely} determined by its diagonalization under
some conditions on the domain as well as on the range. Suppose that
$G$ is a commutative semigroup and $S$ is a commutative group. The
action of the {\emph{difference operator}} $\Delta$ on a function
$f\colon G\to S$ is defined by the formula
\[\Delta_y f(x)=f(x+y)-f(x)
\qquad
\left(x, y\in G\right). \]
Note that the addition in the argument of the function is the
operation of the semigroup $G$ and the subtraction means the inverse
of the operation of the group $S$.

\begin{thm}[Polarization formula]\label{Thm_polarization}
 Suppose that $G$ is a commutative semigroup, $S$ is a commutative group, $n\in \mathbb{N}$.
 If $A\colon G^{n}\to S$ is a symmetric, $n$-additive function, then for all
 $x, y_{1}, \ldots, y_{m}\in G$ we have
 \[
  \Delta_{y_{1}, \ldots, y_{m}}A^{\ast}(x)=
  \left\{
  \begin{array}{rcl}
   0 & \text{ if} & m>n \\
   n!A(y_{1}, \ldots, y_{m}) & \text{ if}& m=n.
  \end{array}
  \right.
 \]

\end{thm}

\begin{cor}
 Suppose that $G$ is a commutative semigroup, $S$ is a commutative group, $n\in \mathbb{N}$.
 If $A\colon G^{n}\to S$ is a symmetric, $n$-additive function, then for all $x, y\in G$
 \[
  \Delta^{n}_{y}A^{\ast}(x)=n!A^{\ast}(y).
\]
\end{cor}

\begin{lem}
\label{mainfact}
  Let $n\in \mathbb{N}$ and suppose that the multiplication by $n!$ is surjective in the commutative semigroup $G$ or injective in the commutative group $S$. Then for any symmetric, $n$-additive function $A\colon G^{n}\to S$, $A^{\ast}\equiv 0$ implies that
  $A$ is identically zero, as well.
\end{lem}

\begin{dfn}
 Let $G$ and $S$ be commutative semigroups, a function $p\colon G\to S$ is called a \emph{generalized polynomial} from $G$ to $S$ if it has a representation as the sum of diagonalizations of symmetric multi-additive functions from $G$ to $S$. In other words, a function $p\colon G\to S$ is a 
 generalized polynomial if and only if, it has a representation 
 \[
  p= \sum_{k=0}^{n}A^{\ast}_{k}, 
 \]
where $n$ is a nonnegative integer and $A_{k}\colon G^{k}\to S$ is a symmetric, $k$-additive function for each 
$k=0, 1, \ldots, n$. In this case we also say that $p$ is a generalized polynomial \emph{of degree at most $n$}. 

Let $n$ be a nonnegative integer, functions $p_{n}\colon G\to S$ of the form 
\[
 p_{n}= A_{n}^{\ast}, 
\]
where $A_{n}\colon G^{n}\to S$ are the so-called \emph{generalized monomials of degree $n$}. 
\end{dfn}

\begin{rem}
 Obviously, generalized monomials 
of degree $0$ are constant functions and generalized monomials of degree $1$ are additive functions. 

 Furthermore, generalized monomials of degree $2$ will be termed \emph{quadratic functions}. 
\end{rem}

\subsection{Polynomial functions}

As Laczkovich \cite{Lac04} enlightens, on groups there are several polynomial notions. One of them is that we introduced in subsection 1.2, that is the notion of generalized polynomials. 
As we will see in the forthcoming sections, not only this notion, but also that of \emph{(normal) polynomials} will be important. The definitions and results recalled here can be found in \cite{Sze91}. 

Throughout this subsection $G$ is assumed to be  a commutative group (written additively).

\begin{dfn}
{\it Polynomials} are elements of the algebra generated by additive
functions over $G$. Namely, if $n$ is a positive integer,
$P\colon\mathbb{C}^{n}\to \mathbb{C}$ is a (classical) complex
polynomial in
 $n$ variables and $a_{k}\colon G\to \mathbb{C}\; (k=1, \ldots, n)$ are additive functions, then the function
 \[
  x\longmapsto P(a_{1}(x), \ldots, a_{n}(x))
 \]
is a polynomial and, also conversely, every polynomial can be
represented in such a form.
\end{dfn}

\begin{rem}
 For the sake of easier distinction, at some places polynomials will be called normal polynomials. 
\end{rem}

\begin{rem}
 We recall that the elements of $\mathbb{N}^{n}$ for any positive integer $n$ are called
 ($n$-dimensional) \emph{multi-indices}.
 Addition, multiplication and inequalities between multi-indices of the same dimension are defined component-wise.
 Further, we define $x^{\alpha}$ for any $n$-dimensional multi-index $\alpha$ and for any
 $x=(x_{1}, \ldots, x_{n})$ in $\mathbb{C}^{n}$ by
 \[
  x^{\alpha}=\prod_{i=1}^{n}x_{i}^{\alpha_{i}}
 \]
where we always adopt the convention $0^{0}=1$. We also use the
notation $\left|\alpha\right|= \alpha_{1}+\cdots+\alpha_{n}$. With
these notations any polynomial of degree at most $N$ on the
commutative semigroup $G$ has the form
\[
 p(x)= \sum_{\left|\alpha\right|\leq N}c_{\alpha}a(x)^{\alpha}
 \qquad
 \left(x\in G\right),
\]
where $c_{\alpha}\in \mathbb{C}$ and $a=(a_1, \dots, a_n) \colon
G\to \mathbb{C}^{n}$ is an additive function. Furthermore, the
\emph{homogeneous term of degree $k$} of $p$ is
\[
 \sum_{\left|\alpha\right|=k}c_{\alpha}a(x)^{\alpha} .
\]
\end{rem}

\begin{lem}[Lemma 2.7 of \cite{Sze91}]\label{L_lin_dep}
 Let $G$ be a commutative group,
 $n$ be a positive integer and let
 \[
  a=\left(a_{1}, \ldots, a_{n}\right),
 \]
where $a_{1}, \ldots, a_{n}$ are linearly independent complex valued
additive functions defined on $G$. Then the monomials
$\left\{a^{\alpha}\right\}$ for different multi-indices are linearly
independent.
\end{lem}

\begin{dfn}
A function $m\colon G\to \mathbb{C}$ is called an \emph{exponential}
function if it satisfies
\[
 m(x+y)=m(x)m(y)
 \qquad
 \left(x,y\in G\right).
\]
Furthermore, on an  \emph{exponential polynomial} we mean a linear
combination of functions of the form $p \cdot m$, where $p$ is a
polynomial and $m$ is an exponential function.
\end{dfn}

\begin{dfn}
 Let $G$ be an Abelian group and $V\subseteq \mathbb{C}^G$ a set of functions. We say that $V$ is {\it translation invariant} if for every $f\in V$ the function $\tau_{g}f\in V$ also holds for all $g\in G$, where
 \[
  \tau_{g}f(h)= f(h+g)
  \qquad
  \left(h\in G\right).
 \]
 \end{dfn}

 In view of Theorem 10.1 of Sz\'ekelyhidi \cite{Sze91}, any finite dimensional translation invariant linear 
 space of complex valued functions on a commutative group consists of exponential polynomials. 
 This implies that if $G$ is a commutative group, then any function 
 $f\colon G\to \mathbb{C}$, satisfying the functional equation 
 \[
  f(x+y)= \sum_{i=1}^{n}g_{i}(x)h_{i}(y) 
  \qquad 
  \left(x, y\in G\right)
 \]
for some positive integer $n$ and functions $g_{i}, h_{i}\colon G\to \mathbb{C}$ ($i=1, \ldots, n$), 
is an exponential polynomial of degree at most $n$.

This enlightens the connection between generalized polynomials and polynomials. It is easy to see that 
each polynomial, that is, any function of the form 
\[
  x\longmapsto P(a_{1}(x), \ldots, a_{n}(x)), 
 \]
where $n$ is a positive integer,
$P\colon\mathbb{C}^{n}\to \mathbb{C}$ is a (classical) complex
polynomial in
 $n$ variables and $a_{k}\colon G\to \mathbb{C}\; (k=1, \ldots, n)$ are additive functions, is a generalized polynomial. The converse however is in general not true. A complex-valued generalized polynomial $p$ defined on a commutative group $G$ is a polynomial \emph{if and only if} its variety (the linear space spanned by its translates) is of \emph{finite} dimension. 
To make the situation more clear, here we also recall Theorem 13.4 from Sz\'ekelyhidi \cite{Sze14}. 
 
 \begin{thm}\label{thm_torsion}
  The torsion free rank of a commutative group is finite \emph{if and only if} every generalized polynomial on the group is a polynomial. 
 \end{thm}

\section{Results}

In this section $\mathbb{F}$ is assumed to be a field with $\mathrm{char}(\mathbb{F})=0$.  
Let further $n$ be a positive integer and $P\in \mathbb{F}[x]$ and $Q\in \mathbb{C}[x]$ be polynomials. Our aim is to prove characterization theorems for generalized polynomials $f\colon \mathbb{F}\to \mathbb{C}$ of degree at most $n$ that also fulfill equation 
\[
 f(P(x))= Q(f(x))
\]
for each $x\in \mathbb{F}$.

Before presenting the results of this paper, we note that related problems have already been considered by Z.~Boros and E.~Garda--M\'{a}ty\'{a}s in \cite{BorGar18, BorGar20, Gar19} and also by M.~Amou in \cite{Amo20}. In these papers the authors consider monomial functions $f, g\colon\mathbb{R}\to \mathbb{R}$ of degree $n$, where $n\in \mathbb{N}$, $n\geq 2$, which satisfy the conditional equation 
\[
 x^{n}f(y)=y^{n}g(x)
\]
for all points $(x, y)$ on a specified planar curve. 

Roughly speaking, the following lemmata tell us that the problem investigated in this paper is meaningful in the sense that for any positive integer $n$, there \emph{do exist} generalized polynomials of degree at most $n$ that satisfy the above identity.

\begin{lem}\label{lemma3}
 Let $\mathbb{F}$ be a field with $\mathrm{char}(\mathbb{F})=0$, $n$ be a positive integer and 
 $\varphi_{1}, \ldots, \varphi_{n}\colon \mathbb{F}\to \mathbb{C}$ be homomorphisms. Define the function $f$ on 
 $\mathbb{F}$ by 
 \[
  f(x)= \varphi_{1}(x)\cdots \varphi_{n}(x)
  \qquad 
  \left(x\in \mathbb{F}\right). 
 \]
Then the following statements hold true. 
\begin{enumerate}[(i)]
 \item The function $f\colon \mathbb{F}\to \mathbb{C}$ is a generalized polynomial of degree $n$. 
 \item The function $f\colon \mathbb{F}\to \mathbb{C}$ is a polynomial of degree $n$.
 \item For any positive integer $k$, we have 
 \[
  f(x^{k})=f(x)^{k} 
  \qquad 
  \left(x\in \mathbb{F}\right). 
 \] 
\end{enumerate}
\end{lem}

\begin{proof}
 Firstly, observe that the function $f$ defined on $\mathbb{F}$ by 
 \[
  f(x)= \varphi_{1}(x)\cdots \varphi_{n}(x)
  \qquad 
  \left(x\in \mathbb{F}\right)
 \]
is the trace of the symmetric $n$-additive mapping $F$ defined by 
\[
 F(x_{1}, \ldots, x_{n}) = \frac{1}{n!}\sum_{\sigma\in \mathscr{S}_{n}}\varphi_{1}(x_{\sigma(1)})\cdots \varphi_{n}(x_{\sigma(n)}) 
 \qquad 
 \left(x \in \mathbb{F}\right). 
\]
Thus $f$ is a generalized monomial of degree $n$. 
Secondly, in view of the definition of the function $f$ we immediately get that it is also a monomial of degree $n$. Indeed, let 
\[
 P(x_{1}, \ldots, x_{n})= x_{1}\cdots x_{n} 
 \qquad 
 \left(x\in \mathbb{C}\right)
\]
and then $f$ can be written as $f= P\circ \varphi$, where $\varphi\colon \mathbb{F}\to \mathbb{C}^{n}$ is 
\[
 \varphi(x)= \left(\varphi_{1}(x), \ldots, \varphi_{n}(x)\right) 
 \qquad 
 \left(x\in \mathbb{F}\right). 
\]
Thirdly, recall that in case $\varphi$ is a homomorphism between $\mathbb{F}$ and $\mathbb{C}$, then 
we also have 
\[
\varphi(x^{k}) = \varphi(x)^{k}
\]
for each $x\in \mathbb{K}$. Therefore, 
\[
 f(x^{k}) = \varphi_{1}(x^{k})\cdots \varphi_{n}(x^{k}) = \varphi_{1}(x)^{k}\cdots \varphi_{n}(x)^{k} 
 =
 \left(\varphi_{1}(x)\cdots\varphi_{n}(x)\right)^{k}
 = f(x)^{k} 
 \qquad 
 \left(x\in \mathbb{F}\right). 
\]
\end{proof}

\begin{lem}\label{lemma4}
 Let $\mathbb{F}\subset \mathbb{C}$ be a field
 and $d\colon \mathbb{F}\to \mathbb{C}$ be a derivation and consider the function 
 $f\colon \mathbb{F}\to \mathbb{C}$ defined by 
 \[
  f(x)= d(x^{n}) 
  \qquad 
  \left(x\in \mathbb{F}\right). 
 \]
Then the following statements are satisfied. 
\begin{enumerate}[(i)]
 \item The function $f\colon \mathbb{F}\to \mathbb{C}$ is a generalized polynomial of degree $n$. 
 \item The function $f\colon \mathbb{F}\to \mathbb{C}$ is a polynomial of degree $n$.
 \item For any positive integer $k\geq 2$, we have 
 \[
  f(x^{k})=kx^{(k-1)n}f(x)
  \qquad 
  \left(x\in \mathbb{F}\right). 
 \] 
\end{enumerate}
\end{lem}

\begin{proof}
 Let $\mathbb{F}\subset \mathbb{C}$ be a field and  
 $d\colon \mathbb{F}\to \mathbb{C}$ be a derivation and let us consider the function 
 $f\colon \mathbb{F}\to \mathbb{C}$ defined through 
 \[
  f(x)= d(x^{n}) 
  \qquad 
  \left(x\in \mathbb{F}\right). 
 \]
Observe that in this case the function $F_{n}\colon \mathbb{F}^{n}\to \mathbb{C}$ 
defined by 
\[
 F_{n}(x_{1}, \ldots, x_{n})= d\left(x_{1}\cdots x_{n}\right) 
 \qquad 
 \left(x_{1}, \ldots, x_{n}\in \mathbb{F}\right)
\]
is a symmetric and $n$-additive function. Furthermore, its trace is $f$, showing that $f$ is a generalized monomial of degree $n$. 

At the same time, since $d$ is a derivation, we also have that 
\[
 d(x^{n})= nx^{n-1}d(x) 
 \qquad 
 \left(x\in \mathbb{F}\right), 
\]
yielding that 
\[
 f(x)= nx^{n-1}d(x)= P(x, d(x)) 
 \qquad 
 \left(x\in \mathbb{F}\right), 
\]
with the two-variable complex polynomial
\[
 P(x, y)= nx^{n-1}y 
 \qquad 
 \left(x, y\in \mathbb{C}\right). 
\]
Therefore, $f$ is a monomial of degree $n$. 

Finally, let $k\geq 2$  be a positive integer. Then 
\[
 f(x^{k})= d\left((x^{k})^{n}\right)= d(x^{kn})= kn x^{kn-1}d(x)
 =
 kx^{(k-1)n}\cdot \left(nx^{n-1}d(x)\right)= kx^{(k-1)n}f(x) 
 \qquad 
 \left(x\in \mathbb{F}\right). 
\]
\end{proof}

\begin{lem}\label{lemma5}
 Let $n$ and $k$ be positive integers, $\alpha\in \mathbb{N}^{n}$ be an n-dimensional multi-index, $\mathbb{F}$ be a field with $\mathrm{char}(\mathbb{F})=0$ and 
 $a_{1}, \ldots, a_{n}\colon \mathbb{F} \allowbreak\to \mathbb{C}$ be additive functions and 
 $a=(a_{1}, \ldots, a_{n})$. Then the mapping 
 \[
  \mathbb{F}\ni x\longmapsto a^{\alpha}(x^{k})
 \]
is a generalized monomial of degree $|\alpha|\cdot k$. 
\end{lem}

\begin{proof}
 Let $n$ and $k$ be positive integers, $\alpha\in \mathbb{N}^{n}$ be an $n$-dimensional multi-index and 
 $a_{1}, \ldots, a_{n}\colon \mathbb{F} \allowbreak\to \mathbb{C}$ be additive functions. 
 If $\alpha= (\alpha_{1}, \ldots, \alpha_{n})$, then the function 
 \[
 \mathbb{F}\ni  x\longmapsto a_{1}^{\alpha_{1}}(x^{k})\cdots a_{n}^{\alpha_{n}}(x^{k}) 
 \]
is the trace of the $|\alpha|\cdot k$-additive function $F$ defined by 
\[
 F(x_{1, 1}, \ldots, x_{\alpha_{n}, k})= 
 \prod_{i=1}^{n}\prod_{j=1}^{\alpha_{i}}a_{i}(x_{j, 1}\cdots x_{j, k}). 
\]
\end{proof}

Let $\mathbb{F}$ be a field and denote $\mathbb{F}^{\times}$ the multiplicative group of the nonzero elements of $\mathbb{F}$. 
Obviously, every (normal) polynomial $p\colon \mathbb{F}^{\times}\to \mathbb{C}$ is a generalized polynomial, too. Furthermore, in view of Theorem \ref{thm_torsion} we know that these are the only generalized polynomials if the torsion free rank of $\mathbb{F}^{\times}$ is finite. For example if we take 
$\mathbb{F}=\mathbb{Q}$ (the field of the rationals), then this is the case. While for larger fields, the same does not hold in general. Despite Theorem \ref{thm_torsion} gives an elegant criterion for the above problem, with the aid of this result it is not too easy to imagine how generalized polynomials look like. From one hand, this is the purpose of the remark below.

\begin{rem}\label{rem3}
 Notice that the above lemma cannot be in general strengthened. To see this let 
 $a\colon \mathbb{F}\to \mathbb{C}$ be a non-identically zero additive function. 
 As the following proposition shows, the mapping 
 \[
  \mathbb{F} \ni x\longmapsto a(x^{2})
 \]
 is a generalized monomial of degree two that is not necessarily a (normal) monomial. 
 \end{rem}
 
 \begin{prop}
  Let $\mathbb{F}\subset \mathbb{C}$ be a field and $a\colon \mathbb{F}\to\mathbb{C}$ be a non-identically zero additive function. The mapping 
  \[
  \mathbb{F} \ni x\longmapsto a(x^{2})
  \]
  is a monomial of degree two if and only if 
  \[
   a(x)= \varphi(d(x))+a(1)\cdot \varphi(x)
   \qquad 
   \left(x\in \mathbb{F}\right)
  \]
or 
\[
 a(x)= \alpha \varphi_{1}(x)+\beta \varphi_{2}(x) 
 \qquad 
 \left(x\in \mathbb{F}\right), 
\]
where $\alpha, \beta$ are complex constants, $d\colon \mathbb{F}\to \mathbb{C}$ is a non-identically zero derivation and $\varphi, \varphi_{1}, \varphi_{2}\colon \mathbb{F}\to \mathbb{C}$ are homorphisms such that $\varphi_{1}$ and $\varphi_{2}$ are linearly independent. 
 \end{prop}

\begin{proof} 
 Let $\mathbb{F}\subset \mathbb{C}$ be a field and $a\colon \mathbb{F}\to\mathbb{C}$ be a non-identically zero additive function.
Assume that the mapping appearing in the proposition is a monomial of degree two. Then there exist linearly independent additive functions 
$a_{1}, a_{2}\colon \mathbb{F}\to \mathbb{C}$ and complex constants $\alpha_{i, j}$, $i, j=1, 2$ such that 
\[
 a(x^{2})= \alpha_{1, 1}a_{1}(x)^{2}+(\alpha_{1, 2}+\alpha_{2, 1})a_{1}(x)a_{2}(x)+\alpha_{2, 2}a_{2}(x)^{2}
\]
for each $x\in \mathbb{F}$. Since both sides of the above identity are traces of symmetric bi-additive functions, we can use the Polarization Formula to get that 
\[
 a(xy)= \alpha_{1, 1}a_{1}(x)a_{1}(y)+\frac{\alpha_{1, 2}+\alpha_{2, 1}}{2}\left(a_{1}(x)a_{2}(y)+a_{1}(y)a_{2}(x)\right)+\alpha_{2, 2}a_{2}(x)a_{2}(y)
\]
is fulfilled by any $x, y\in \mathbb{F}$. 
After some rearrangement, we have 
\[
 a(xy)
 = a_{1}(x) \cdot \left(\alpha_{1, 1}a_{1}(y)+\frac{\alpha_{1, 2}+\alpha_{2, 1}}{2}a_{2}(y)+
 \right)
 +a_{2}(x)\cdot \left(\frac{\alpha_{1, 2}+\alpha_{2, 1}}{2} a_{1}(y)+\alpha_{2, 2}a_{2}(y)\right)
\]
for all $x, y\in \mathbb{F}$, which is a Levi-Civita equation on $\mathbb{F}^{\times}$ (on the multiplicative group of the non-zero elements of $\mathbb{F}$). Applying the result of \cite[pages 93--94]{Sze91}, we deduce that there exist complex constants $\alpha, \beta$, an additive function 
$l\colon \mathbb{F}^{\times}\to \mathbb{C}$ and exponentials $m, m_{1}, m_{2}\colon \mathbb{F}^{\times}\to \mathbb{C}$, $m_{1}$ and $m_{2}$ are linearly independent such that 
\[
 a(x)= (\alpha l(x)+\beta) m(x) 
 \qquad 
 \left(x\in \mathbb{F}^{\times}\right)
\]
or 
\[
 a(x)= m_{1}(x)+m_{2}(x) 
 \qquad 
 \left(x\in \mathbb{F}^{\times}\right). 
\]
Recall that on the group $\mathbb{F}^{\times}$ multiplication as group operation is considered. Therefore these functions fulfill the following identities: 
\[
 l(xy)= l(x)+l(y) 
 \qquad 
 \left(x, y\in \mathbb{F}^{\times}\right)
\]
and 
\[
 m(xy)= m(x)m(y) 
 \qquad 
 \left(x, y\in \mathbb{F}^{\times}\right)
\]
as well as 
\[
 m_{i}(xy)= m_{i}(x)m_{i}(y) 
 \qquad 
 \left(x, y\in \mathbb{F}^{\times}, i=1, 2\right). 
\]
At the same time, the mapping $a$ was assumed to be additive, thus in the above representation we have 
\[
   a(x)=\varphi(d(x))+ a(1)\cdot\varphi(x) 
   \qquad 
   \left(x\in \mathbb{F}\right)
  \]
or 
\[
 a(x)= \alpha \varphi_{1}(x)+\beta \varphi_{2}(x) 
 \qquad 
 \left(x\in \mathbb{F}\right), 
\]
where $\alpha, \beta$ are complex constants, $d\colon \mathbb{F}\to \mathbb{C}$ is a non-identically zero derivation and $\varphi, \varphi_{1}, \varphi_{2}\colon \mathbb{F}\to \mathbb{C}$ are homomorphisms such that $\varphi_{1}$ and $\varphi_{2}$ are linearly independent. 

The converse is obvious. 
\end{proof}

\begin{rem}
 The proof of the previous proposition can be a starting point of further investigations. More precisely, if $n$ is a positive integer, $a\colon \mathbb{F}\to \mathbb{C}$ is a non-identically zero additive function and the mapping 
  \[
  \mathbb{F} \ni x\longmapsto a(x^{2})
  \]
  is a monomial of degree $n$, then there exist linearly independent additive functions 
  $a_{1}, \ldots, a_{n}\colon \mathbb{F}\to \mathbb{C}$ and complex constants $\alpha_{i, j}$, $i, j=1, \ldots, n$ such that 
  \[
   a(x^{2})= \sum_{i, j=1}^{n}\alpha_{i, j}a_{i}(x)a_{j}(x) 
   \qquad 
   \left(x\in \mathbb{F}\right). 
  \]
Since both sides of the above identity are traces of symmetric bi-additive functions, we can use the Polarization Formula to get that
\[
 a(xy)= \sum_{i, j=1}^{n}\frac{\alpha_{i, j}+\alpha_{j, i}}{2}\left(a_{i}(x)a_{j}(y)+a_{j}(x)a_{i}(y)\right)
 = \sum_{i=1}^{n}\widetilde{a_{i}}(x)\widetilde{a_{j}}(y)
\]
for all $x, y\in \mathbb{F}^{\times}$ which is a Levi-Civita equation on the group $\mathbb{F}^{\times}$. In other words, the mapping $a$ as a function restricted to the multiplicative group $\mathbb{F}^{\times}$, is a normal exponential polynomial of degree $n$. In view of the results of \cite[page 43 and page 79]{Sze91}, 
\[
 a(x)= \sum_{j=1}^{k}P_{j}\left(l_{j, 1}(x), l_{j, 2}(x), \ldots, l_{j, n_{j}-1}(x)\right)m_{j}(x) 
 \qquad 
 \left(x\in \mathbb{F}^{\times}\right), 
\]
where $k, n_{1}, \ldots, n_{k}$ are positive integers, $m_{1}, \ldots, m_{k}$ are different, non-zero com\-plex-valued exponentials on the group $\mathbb{F}^{\times}$, further 
$\left\{l_{j, 1}, \ldots, l_{j, n_{j}-1} \right\}$ are sets of linearly independent, com\-plex-valued additive functions defined on $\mathbb{F}^{\times}$ for $j=1, \ldots, k$ and 
$P_{j}\colon \mathbb{C}^{n_{j}-1}\to \mathbb{C}$ are complex polynomials of degree at most $n_{j}-1$ and in $n_{j}-1$ variables for each $j=1, \ldots, k$. 

We conjecture that as a continuation, a result of Kiss--Laczkovich \cite{KisLac18} might be useful. According to Theorem 1.1 of this paper, $a\colon \mathbb{F}\to \mathbb{C}$ is an additive function with $a(1)=0$ and 
$D/j$, as a map from the group $\mathbb{F}^{\times}$ to $\mathbb{C}$, is a generalized polynomial of degree at most $n$ if and only if $a$ is a derivation of order at most $n$. Here $j$ denotes the identity map from $\mathbb{F}$ to $\mathbb{C}$. 
\end{rem}

\begin{ex}
 Let $d\colon \mathbb{C}\to \mathbb{C}$ be a non-identically zero derivation. Due to Theorem \ref{T14.2.1} such a mapping does exist. Furthermore, due to Theorem 14.5.1 of \cite{Kuc}, there exist non-trivial endomorphisms of $\mathbb{C}$ (trivial endomorphisms are meant the identically zero and the identity mapping, resp.). Let $\varphi\colon \mathbb{C}\to \mathbb{C}$ be a non-trivial endomorphism. Then the mapping $a$ defined on $\mathbb{C}$ by 
 \[
  a(x)= d(x)+\varphi(x) 
  \qquad 
  \left(C\right)
 \]
is clearly additive and we have 
\[
 a(x^2)= d(x^{2})+\varphi(x^{2})= 2xd(x)+\varphi(x)^{2} 
 \qquad 
 \left(x\in \mathbb{F}\right), 
\]
showing that the mapping 
\[
 \mathbb{C} \ni x \longmapsto a(x^{2})
\]
is a generalized polynomial of degree exactly two and it is a (normal) polynomial of degree exactly 
three. 

Note that the situation changes if the domain is the real field. Indeed, all homomorphisms 
$\varphi\colon \mathbb{R}\to \mathbb{C}$ are trivial, i.e. they are either identically zero or they coincide with the identity map. Therefore if $d\colon \mathbb{R}\to \mathbb{C}$ is a non-identically zero derivation and we define 
\[
 a(x)= d(x)+x 
 \qquad 
 \left(x\in \mathbb{R}\right),
\]
then the mapping \[
 \mathbb{C} \ni x \longmapsto a(x^{2})
\]
is a polynomial of degree exactly two. 
\end{ex}

\begin{rem}
From the above thoughts we obtain that if $a\colon \mathbb{F}\to \mathbb{C}$ is an additive function such that the mapping 
\[
 \mathbb{F} \ni x \longmapsto a(x^{2})
\]
is a generalized monomial which is not a monomial, then for all positive integer $k\geq 2$, the mapping 
\[
 \mathbb{F} \ni x \longmapsto a(x^{k})
\]
is also a generalized monomial which is not a monomial. 

Indeed, observe that the above mapping is the trace of the symmetric and $k$-additive function $A_{k}$ defined on $\mathbb{F}^{k}$ by 
\[
 A_{k}(x_{1}, \ldots, x_{k})= a(x_{1}\cdots x_{k}) 
 \qquad 
 \left(x_{1}, \ldots, x_{k}\in \mathbb{F}\right). 
\]

If there would exist a positive integer $k>2$ such that the mapping 
\[
 \mathbb{F} \ni x \longmapsto a(x^{k})
\]
would be a monomial, then all the translates of this mapping would also be generalized monomials, as well. 
By the Polarization formula, we have  
\[
 \Delta_{y}^{2}\Delta^{k-2}_{1}a(x^{k})= k!A(y, y, 1, \ldots, 1)= k!a(y^{2}) 
 \qquad 
 \left(x, y\in \mathbb{F}\right). 
\]
From this we would deduce that the mapping 
\[
 \mathbb{F} \ni y \longmapsto a(y^{2})
\]
is a monomial of degree two, which is a contradiction. 
\end{rem}

\begin{lem}\label{lemma_6}
 Let $k$ and $n$ be positive integers and $f\colon \mathbb{F}\to \mathbb{C}$ be a generalized monomial of degree n, where $\mathbb{F}$ is assumed to be a field with $\mathrm{char}(\mathbb{F})=0$. Then the mapping 
 \[
  \mathbb{F} \ni x \longmapsto f(x^{k})
 \]
is a generalized monomial of degree $n\cdot k$. 
\end{lem}

\begin{proof}
 Since $f\colon \mathbb{F}\to \mathbb{C}$ is a generalized monomial of degree $n$, there exists an 
 $n$-additive function $F_{n}\colon \mathbb{F}^{n}\to \mathbb{C}$ such that its trace is the function $f$. In this case, the mapping 
 $F_{nk}\colon \mathbb{F}^{nk}\to \mathbb{C}$ defined by 
 \begin{multline*}
  F_{nk}(x_{1, 1}, \ldots, x_{n, k}) =
  F_{n}(x_{1, 1} \cdots x_{1, k}, x_{2, 1} \cdots x_{2, k}, \ldots, x_{n, 1} \cdots x_{n, k})
  \\
  \left(x_{i, j}\in \mathbb{F}, i=1, \ldots, n, j=1, \ldots, k\right)
 \end{multline*}
is an $(n\cdot k)$-additive function whose trace is 
\[
 F_{nk}(x, \ldots, x)= F_{n}(x^{k}, \ldots, x^{k})= f(x^{k})
 \qquad 
 \left(x\in \mathbb{F}\right). 
\]
\end{proof}

\begin{rem}
 The above lemma with a different proof can be found among others in \cite{AicMoo21}. 
\end{rem}

Now we turn to deal with the main problem of this paper, which is the following. 
Assume $\mathbb{F}$ to be a field. 
Let further $P\in \mathbb{F}[x]$ and $Q\in \mathbb{C}[x]$ be polynomials. Our aim is to prove characterization theorems for quadratic functions $f\colon \mathbb{F}\to \mathbb{C}$ that also fulfill equation 
\[
 f(P(x))= Q(f(x))
\]
for each $x\in \mathbb{F}$.

As for the difficulty of the problem, it is an important condition that the function $f$ is supposed to be quadratic (that is, it is a \emph{generalized} monomial of degree two). As the statement below shows, 
although the result is the same, the proof is much easier for normal monomials of degree two. 

\begin{prop}\label{Prop1}
 Let $\mathbb{F}$ be a field with $\mathrm{char}(\mathbb{F})=0$ and $f\colon \mathbb{F}\to \mathbb{C}$ be a quadratic function that can be represented as 
 \[
  f(x)= a_{1}(x)a_{2}(x) 
  \qquad 
  \left(x\in \mathbb{F}\right)
 \]
 with the aid of the additive functions $a_{1}, a_{2}\colon \mathbb{F}\to \mathbb{C}$. 
Then equation 
 \begin{equation}\label{Eq1_simp}
  f(x^{2})= f(x)^{2}
 \end{equation}
holds for each $x\in \mathbb{F}$ \emph{if and only if} there exist complex-valued non-trivial homomorphisms  $\varphi_{1}, \varphi_{2}$ defined on $\mathbb{F}$ such that 
\[
 f(x)= f(1)\cdot \varphi_{1}(x)\varphi_{2}(x)
 \qquad 
 \left(x\in \mathbb{F}\right). 
\]
Furthermore, 
\begin{enumerate}[(A)]
 \item either $f(1)=0$ and then $f$ is identically zero;
 \item or $f(1)=1$. 
\end{enumerate}
\end{prop}

\begin{proof}
 Assume $\mathbb{F}$ be a field and $f\colon \mathbb{F}\to \mathbb{C}$ be a quadratic function that can be represented as 
 \[
  f(x)= a_{1}(x)a_{2}(x) 
  \qquad 
  \left(x\in \mathbb{F}\right)
 \]
 with the aid of the additive functions $a_{1}, a_{2}\colon \mathbb{F}\to \mathbb{C}$. 
 Then equation \eqref{Eq1_simp} in terms of the additive functions $a_{1}$ and $a_{2}$ is 
 \[
  a_{1}(x^{2})a_{2}(x^{2})= a_{1}(x)^{2}a_{2}(x)^{2} 
  \qquad 
  \left(x\in \mathbb{F}\right). 
 \]
Since both sides of this identity are traces of symmetric, $4$-additive functions, a simple application of the Polarization Formula leads to 
\begin{multline}\label{Eq1_simp_sym}
 a_{1}(xy)a_{2}(uv)+a_{1}(xu)a_{2}(yv)+a_{1}(xv)a_{2}(yu)
 \\
 +a_{1}(yu)a_{2}(xv)+a_{1}(yv)a_{2}(xu)+a_{1}(ux)a_{2}(xy)
 \\
 =
 a_{1}(x)a_{1}(y)a_{2}(u)a_{2}(v)+a_{1}(x)a_{1}(u)a_{2}(y)a_{2}(v)+a_{1}(x)a_{1}(v)a_{2}(y)a_{2}(u)
 \\
 +a_{1}(y)a_{1}(u)a_{2}(x)a_{2}(v)+a_{1}(y)a_{1}(v)a_{2}(x)a_{2}(u)+a_{1}(u)a_{1}(x)a_{2}(x)a_{2}(y)
 \\
 (x, y, u, v\in \mathbb{F}). 
\end{multline}
From this, with the substitution $x=y=u=v=1$, we immediately get that 
\[
 a_{1}(1)\,a_{2}(1)\,\left(a_{1}(1)\,a_{2}(1)-1\right)=0
\]
and with the substitution $y=u=v=1$, 
\[
 \left(a_{1}(1)\,a_{2}(1)-1\right)\,\left(a_{1}(1)\,a_{2}(x)+a_{2}(
 1)\,a_{1}(x)\right)=0
\]
can be deduced for all $x\in \mathbb{F}$. 
Assume first that $a_{1}(1)\,a_{2}(1)-1\neq 0$, then $a_{1}(1)=0$ or $a_{2}(1)=0$. If $a_{1}(1)=0$, then 
\begin{enumerate}[(A)]
 \item either $a_{2}(1)\neq 0$ and the above identity reduces to 
 \[
  a_{2}(1)a_{1}(x)=0 
  \qquad 
  \left(x\in \mathbb{F}\right). 
 \]
In other words $a_{1}$ is the identically zero function. In this case $f$ is identically zero, too. 
\item or $a_{2}(1)=0$ and from \eqref{Eq1_simp_sym} we get that 
\begin{multline*}
 a_{1}(1)\,a_{2}(x\,y)+a_{2}(1)\,a_{1}(x\,y)+\left(\left(2-2\,a_{1}(
 1)\,a_{2}(1)\right)\,a_{1}(x)-a_{1}(1)^2\,a_{2}(x)\right)\,a_{2}(y)
 \\
 +
 \left(\left(2-2\,a_{1}(1)\,a_{2}(1)\right)\,a_{2}(x)-a_{2}(1)^2\,a_{
 1}(x)\right)\,a_{1}(y)=0
\end{multline*}
for all $x, y\in \mathbb{F}$, that is, 
\[
 a_{1}(x)\,a_{2}(y)+a_{2}(x)\,a_{1}(y)=0 
 \qquad 
 \left(x, y\in \mathbb{F}\right). 
\]
From this, 
\[
 a_{1}(x)a_{2}(x)=0
\]
follows for all $x\in \mathbb{F}$, that is, $f$ is the identically zero function.
\end{enumerate}
Therefore, from now on $a_{1}(1)\,a_{2}(1)-1=0$ can be supposed. Note that without the loss of generality we can (and we do) assume that $a_{1}(1)=a_{2}(1)=1$. In this case equation 
\eqref{Eq1_simp_sym} implies that 
\[
 a_{2}(x\,y)+a_{1}(x\,y)=a_{2}(x)\,a_{2}(y)+a_{1}(x)\,a_{1}(y)
\]
for all $x, y\in \mathbb{F}$. From this we deduce that either 
\begin{enumerate}[(A)]
 \item $\left\{a_{1}, a_{2}\right\}$ is linearly dependent, that is 
 \[
  a_{2}=ca_{1}(x) 
  \qquad 
  \left(x\in \mathbb{F}\right), 
 \]
from which we get that $a_{1}\equiv a_{2}$, since $a_{1}(1)=a_{2}(1)$. If so, then equation \eqref{Eq1_simp} reduces to 
\[
 a(x^{2})^{2}= \kappa a(x)^{4} 
 \qquad 
 \left(x\in \mathbb{F}\right), 
\]
that is $a$ is a homomorphism and 
\[
 f(x)= f(1)a(x)^{2} 
 \qquad 
 \left(x\in \mathbb{F}\right). 
\]
\item or $\left\{a_{1}, a_{2}\right\}$ is linearly independent. Then due to \cite[pages 93-94]{Sze91}
\[
 a_{i}(x)= (\alpha_{i}l(x)+\beta_{i})m(x) 
 \qquad 
 \left(x\in \mathbb{F}^{\times}, i=1, 2\right)
\]
or 
\[
 a_{i}(x)= \alpha_{i}m_{1}(x)+\beta_{i}m_{2}(x) 
 \qquad 
 \left(x\in \mathbb{F}^{\times}, i=1, 2\right), 
\]
where $l\colon \mathbb{F}^{\times}\to \mathbb{C}$ is a logarithmic function, 
$m, m_{1}, m_{2}\colon \mathbb{F}^{\times}\to \mathbb{C}$ are multiplicative functions and 
$\alpha_{1}, \alpha_{2}$ and $\beta_{1}, \beta_{2}$ are complex constants. 

Substituting the first form into equation \eqref{Eq1_simp}, $\alpha_{1}=0$ and $\alpha_{2}=0$. That is, $a_{1}$ and $a_{2}$ are constant multiples of a (nonzero) homomorphism. This means that there exists linearly independent homomorphisms $\varphi_{1}, \varphi_{2}\colon \mathbb{F}\to \mathbb{C}$ such that 
\[
 f(x)= f(1)\varphi_{1}(x)\varphi_{2}(x) 
 \qquad 
 \left(x\in \mathbb{F}\right). 
\]
Finally, if we substitute the second form into \eqref{Eq1_simp} then we get (among others) that 
\[
 \begin{array}{rcl}
 \alpha_{1}\alpha_{2}(1-\alpha_{1}\alpha_{2})&=&0\\
  \beta_{1}\beta_{2}(1-\beta_{1}\beta_{2})&=&0
 \end{array}
\]
If we would have $\beta_{1}\beta_{2}=1$, then the remaining equations, that is, 
\[
 \begin{array}{rcl}
  \alpha_{1}\beta_{1}\beta_{2}^{2}+\alpha_{2}\beta_{1}^{2}\beta_{2}&=&0\\
  \alpha_{1}^{2}\alpha_{2}\beta_{2}+\alpha_{1}\beta_{2}^{2}\beta_{1}&=&0\\
  \alpha_{1}\beta_{2}+\beta_{1}\alpha_{2}-\alpha_{1}^{2}\beta_{2}^{2}-\beta_{1}^{2}\alpha_{2}-4\alpha_{1}\alpha_{2}\beta_{1}\beta_{2}&=&0
 \end{array}
\]
would lead to a contradiction. The same concerns the case $\alpha_{1}\alpha_{2}=1$. 

This means however that $\alpha_{1}\alpha_{2}=0$ and $\beta_{1}\beta_{2}=0$, that is, either $f$ is the 
identically zero function, or there exists homomorphisms $\varphi_{1}, \varphi_{2}\colon \mathbb{F}\to \mathbb{C}$ such that 
\[
 f(x)=f(1)\varphi_{1}(x)\varphi_{2}(x) 
 \qquad 
 \left(x\in \mathbb{F}\right). 
\]
\end{enumerate}
\end{proof}

\begin{rem}
 Roughly speaking the above statement says that among polynomials whose variety is at most two-dimensional, the solutions $f\colon \mathbb{F}\to \mathbb{C}$ of equation \eqref{Eq1_simp} are of the form 
 \[
  f(x)= f(1)\varphi_{1}(x)\varphi_{2}(x) 
  \qquad 
  \left(x\in \mathbb{F}\right), 
 \]
with appropriate homomorphisms $\varphi_{1}, \varphi_{2}\colon \mathbb{F}\to \mathbb{C}$. 

Let now $n$ be a fixed positive integer. 
In the general case, that is, if $f\colon \mathbb{F}\to \mathbb{C}$ is a quadratic function whose variety is at most $n$-dimensional, then we have 
\[
 f(x)=\sum_{p, q=1}^{n}a_{p}(x)a_{q}(x) 
 \qquad 
 \left(x\in \mathbb{F}\right), 
\]
with certain additive functions $a_{1}, \ldots, a_{n}\colon \mathbb{F}\to \mathbb{C}$. 
In this situation equation \eqref{Eq1_simp} can be investigated with an analogous argument as in the proof of Proposition \ref{Prop1}. 
\end{rem}

The starting point of study of equation \eqref{Eq1} among quadratic functions is the theorem below that can be found in \cite{GseKisVin19}. Here we also recall its proof to enlighten the way to the general case. 

\begin{thm}\label{THM_main}
 Let $\mathbb{F}$ be a field with $\mathrm{char}(\mathbb{F})=0$ and $f\colon \mathbb{F}\to \mathbb{C}$ be a quadratic function. 
 Then equation 
 \begin{equation}\label{Eq1}
  f(x^{2})= f(x)^{2}
 \end{equation}
holds for each $x\in \mathbb{F}$ \emph{if and only if} there exist non-trivial homomorphisms $\varphi_{1}, \varphi_{2}\colon \allowbreak \mathbb{F}\to \mathbb{C}$ such that 
\[
 f(x)= f(1)\cdot \varphi_{1}(x)\varphi_{2}(x)
 \qquad 
 \left(x\in \mathbb{F}\right). 
\]
Furthermore, 
\begin{enumerate}[(A)]
 \item either $f(1)=0$ and then $f$ is identically zero;
 \item or $f(1)=1$. 
\end{enumerate}

\end{thm}

\begin{proof}
 Since $f$ is a generalized monomial of degree $2$, there exists a
symmetric bi-additive function $F_{2}\colon \mathbb{F}^{2}\allowbreak
\to \mathbb{C}$ so that
\[
 F_{2}(x, x)= f(x)
 \qquad
 \left(x\in \mathbb{F}\right).
\]
Define the symmetric $4$-additive mapping $F_{4}\colon
\mathbb{F}^{4}\to \mathbb{C}$ through
\begin{multline*}
 F_{4}(x_{1}, x_{2}, x_{3}, x_{4})
 =
 F_{2}(x_{1}x_{2}, x_{3}x_{4})
 +  F_{2}(x_{1}x_{3}, x_{2}x_{4})
 + F_{2}(x_{1}x_{4}, x_{2}x_{3})
 \\
 -F_{2}(x_{1}, x_{2})F(x_{3}, x_{4})
 -F_{2}(x_{1}, x_{3})F(x_{2}, x_{4})
 -F_{2}(x_{1}, x_{4})F(x_{2}, x_{3})
 \\
 \left(x_{1}, x_{2}, x_{3}, x_{4}\in \mathbb{F}\right).
\end{multline*}
Since
\[
 F_{4}(x, x, x, x)= 3 \left(F_{2}(x^{2}, x^{2})-F_{2}(x, x)^{2}\right)=3 \left(f(x^{2})-f(x)^{2}\right)=0
 \qquad
 \left(x\in \mathbb{F}\right),
\]
the mapping $F_{4}$ has to be identically zero on
$\mathbb{F}^{4}$. Therefore, especially
\[
0=F_{4}(1, 1, 1, 1)= 3F_{2}(1, 1)-3F_{2}(1, 1)^{2},
\]
yielding that either $F_{2}(1, 1)=0$ or $F_{2}(1, 1)=1$. Moreover,
\[
 0= F_{4}(x, 1, 1, 1) =
 3F_{2}\left(x , 1\right)-3F_{2}\left(1 , 1\right)F_{2}\left(x , 1\right)
 \qquad
 \left(x\in \mathbb{F}\right),
\]
from which either $F_{2}(1, 1)=1$ or $F_{2}(x, 1)=0$ follows for any $x\in
\mathbb{F}$.

Using that
\[
 0= F_{4}(x, x, 1, 1)
 =
 F_{2}(x^2 , 1)-F_{2}\left(1 , 1\right)\,F_{2}\left(x , x\right)+2F_{2} \left(x , x\right)-2F^2_{2}\left(x , 1\right)
 \qquad 
 \left(x\in \mathbb{F}\right),
\]
we obtain that
\[
 \left(F_{2}(1, 1)-2\right)F_{2}\left(x , x\right)= F_{2}(x^2 , 1)-2F^2_{2}\left(x , 1\right)
 \qquad
 \left(x\in \mathbb{F}\right).
\]
Now, if $F_{2}(1, 1)=0$, then according to the above identities $F_{2}(x,
1)=0 $ would follow for all $x\in \mathbb{F}$. Since $F_{4}(x,
x, 1, 1)=0$ is also fulfilled by any $x\in \mathbb{F}$, this
immediately implies that
\[
 -2f(x)=-2F_{2}(x, x)= F_{2}(x^{2}, 1)-F_{2}(x, 1)^{2}=0
 \qquad
 \left(x\in \mathbb{F}\right),
\]
i.e., $f$ is identically zero.

In case $F_{2}(1, 1)\neq 0$, then necessarily $F_{2}(1, 1)=1$ from which
\[
 -F_{2}(x, x)= F_{2}(x^{2}, 1)-2F_{2}(x, 1)^{2}
 \qquad
 \left(x\in \mathbb{F}\right).
\]
Define the non-identically zero additive function $a\colon
\mathbb{F}\to \mathbb{C}$ by
\[
 a(x)=F_{2}(x, 1)
 \qquad
 \left(x\in \mathbb{F}\right)
\]
to get that
\[
 f(x)=F_{2}(x, x)= -F_{2}(x^{2}, 1)+2F_{2}(x, 1)^{2}= 2a(x)^{2} -a(x^{2})
 \qquad
 \left(x\in \mathbb{F}\right).
\]
Since $F_{4}(x, x, x, x)=0$ has to hold, the additive function
$a\colon \mathbb{F}\to \mathbb{C}$ has to fulfill identity
\begin{equation}\label{Eq3}
 -a(x^4)+a^2(x^2)+4a^2(x)a(x^2)-4a^4(x)=0
 \qquad
 \left(x\in \mathbb{F}\right)
\end{equation}
too.

In what follows, we will show that the additive function $a$ is of a
rather special form.

Indeed,
\[
 0= F_{4}(x, y, z, 1)
 \qquad
 \left(x, y, z\in \mathbb{F}\right)
\]
means that $a$ has to fulfill equation
\[
 a(x)a(yz)+a(y)a(xz)+a(z)a(xy)= 2a(x)a(y)a(z)+a(xyz)
 \qquad
 \left(x, y, z\in \mathbb{F}\right)
\]
Let now $z^{\ast}\in \mathbb{F}$ be arbitrarily fixed to have
\[
 a(x)a(yz^{\ast})+a(y)a(xz^{\ast})+a(z^{\ast})a(xy)= 2a(x)a(y)a(z^{\ast})+a(xyz^{\ast})
 \qquad
 \left(x, y, z\in \mathbb{F}\right).
\]
Define the additive function $A\colon \mathbb{F}\to \mathbb{C}$ by
\[
 A(x)=a(xz^{\ast})-a(z^{\ast})a(x)
 \qquad
 \left(x\in \mathbb{F}\right)
\]
to receive that
\[
 A(xy)=a(x)A(y)+a(y)A(x)
 \qquad
 \left(x, y\in \mathbb{F}\right),
 \]
 which is a special convolution type functional equation.
Due to Theorem 12.2 of \cite{Sze91}, we get that
\begin{enumerate}[(a)]
 \item the function $A$ is identically zero under any choice of $z^{\ast}$, implying that $a$ has to be multiplicative.
 Note that $a$ is additive, too. Thus, for the quadratic mapping $f\colon \mathbb{F}\to \mathbb{C}$ there exists a
 homomorphism $\varphi\colon \mathbb{F}\to \mathbb{C}$ such that
 \[
  f(x)=\varphi(x)^{2}
  \qquad
  \left(x\in \mathbb{F}\right).
 \]

\item or there exists multiplicative functions $m_{1}, m_{2}\colon \mathbb{F}\to \mathbb{C}$
and a complex constant $\alpha$ such that
\[
 a(x)=\frac{m_{1}(x)+m_{2}(x)}{2}
 \qquad
 \left(x\in \mathbb{F}\right)
\]
and
\[
 A(x)=\alpha \left(m_{1}(x)-m_{2}(x)\right)
  \qquad
 \left(x\in \mathbb{F}\right).
\]
Due to the additivity of $a$, in view of the definition of the
mapping $A$, we get that $A$ is additive, too.

This however means that both the maps $m_{1}+m_{2}$ and
$m_{1}-m_{2}$ are additive, from which the additivity of $m_{1}$ and
$m_{2}$ follows, yielding that they are in fact homomorphisms.

Since
\[
 F_{2}(x, x)= f(x)= 2a(x)^{2}-a(x^{2})
 \qquad
 \left(x\in \mathbb{F}\right),
\]
we obtain for the quadratic function $f\colon \mathbb{F}\to
\mathbb{C}$ that there exist homomorphisms $\varphi_{1},
\varphi_{2}\colon \mathbb{F}\to \mathbb{C}$ such that
\[
 f(x)=\varphi_{1}(x)\varphi_{2}(x)
 \qquad
 \left(x\in \mathbb{F}\right).
\]
\end{enumerate}
Summing up, we received the following: identity
\[
 f(x^{2})= f(x)^{2}
 \qquad
 \left(x\in \mathbb{F}\right)
\]
holds for the quadratic function $f\colon \mathbb{F}\to \mathbb{C}$
if and only if there exist homomorphisms $\varphi_{1},
\varphi_{2}\colon \mathbb{F}\to \mathbb{C}$ such that
\[
 f(x)=f(1)\cdot\varphi_{1}(x)\varphi_{2}(x)
 \qquad
 \left(x\in \mathbb{F}\right).
\]
\end{proof}

\begin{cor}
 Let $n\geq 2$ be a positive integer, $\mathbb{F}$ be a field with $\mathrm{char}(\mathbb{F})=0$ and $f\colon \mathbb{F}\to \mathbb{C}$ be a quadratic function. 
 Then equation 
 \begin{equation}\label{Eq2}
  f(x^{n})= f(x)^{n}
 \end{equation}
holds for each $x\in \mathbb{F}$ \emph{if and only if} there exist  homomorphisms $\varphi_{1}, \varphi_{2}\colon \mathbb{F}\to \mathbb{C}$ such that 
\[
 f(x)= f(1)\cdot \varphi_{1}(x)\varphi_{2}(x) 
 \qquad 
 \left(x\in \mathbb{F}\right), 
\]
furthermore either $f(1)=0$ and then $f$ is identically zero, or $f(1)$ is an $(n-1)$\textsuperscript{st} root of unity. 
\end{cor}

\begin{proof}
 In case $n=2$, then in view of the previous theorem, there is nothing to prove. 
 Thus $n>2$ can be supposed subsequently. 
 
Let $\mathbb{F}$ be a field with $\mathrm{char}(\mathbb{F})=0$, $\varphi_{1}, \varphi_{2}\colon \mathbb{F}\to \mathbb{C}$ be homomorphisms and 
$\lambda\in \mathbb{C}$ be such that either $\lambda$ is zero, or it is an $(n-1)$\textsuperscript{st} root of unity. 
Define the function $f\colon \mathbb{F}\to \mathbb{C}$ by 
\[
f(x)= \lambda \varphi_{1}(x)\varphi_{2}(x) 
\qquad 
\left(x\in \mathbb{F}\right). 
\]
Since every homomorphism is additive, Lemma \ref{lemma3} immediately yields that $f$ is a quadratic function and we also have 
\[
 f(x^{n})= \lambda\varphi_{1}(x^{n})\varphi_{2}(x^{n})= \lambda \varphi_{1}(x)^{n}\varphi_{2}(x)^{n}
 = \lambda^{n}\varphi_{1}(x)^{n}\varphi_{2}(x)^{n}=  \left(\lambda \varphi_{1}(x)\varphi_{2}(x)\right)^{n}
 =f(x)^{n}
\]
for each $x\in \mathbb{F}$, since $\lambda= \lambda^{n}$. 
  
Conversely, let $f\colon \mathbb{F}\to \mathbb{C}$ be a quadratic function such that we additionally have that 
  \[
   f(x^{n})= f(x)^{n}
  \]
for all $x\in \mathbb{F}$. Since $f$ is a quadratic function, there exists a symmetric, bi-additive function $F_{2}\colon \mathbb{F}^{2}\to \mathbb{C}$ such that 
\[
 f(x)= F_{2}(x, x) 
 \qquad 
 \left(x\in \mathbb{F}\right). 
\]
Equation \eqref{Eq2} with $x=1$ immediately yields that 
\[
 f(1)=f(1)^{n}, 
\]
that is, $F(1, 1)= f(1)$ is either zero, or it is an $(n-1)$\textsuperscript{st} root of unity. 

Furthermore, equation \eqref{Eq2} in terms of the function $F_{2}$ is 
\[
 F_{2}(x^{n}, x^{n})= F_{2}(x, x)^{n} 
 \qquad 
 \left(x\in \mathbb{F}\right). 
\]
Observe that both the sides of the above identity are traces of symmetric and $2n$-additive functions, namely we have 
\[
 \frac{1}{(2n)!}\sum_{\sigma\in \mathscr{S}_{2n}}F_{2}(x_{\sigma(1)}\cdots x_{\sigma(n)}, x_{\sigma(n+1)} \cdots x_{\sigma(2n)})
 =
 \frac{1}{(2n!)} \sum_{\sigma \in \mathscr{S}_{2n}}F_{2}(x_{\sigma(1)}, x_{\sigma(2)})\cdots F_{2}(x_{\sigma(2n-1)}, x_{\sigma(2n)})
\]
for all $x_{1}, \ldots, x_{2n}\in \mathbb{F}$. 
This identity with the substitution 
\[
 x_{1}= x, \; x_{2}=x \quad x_{i}=1 \quad \text{for } i=3, \ldots, 2n
\]
yields that there are complex constants $\alpha$ and $\beta$ depending only on $F_{2}(1, 1)$ such that 
\[
 f(x)=F_{2}(x, x)= \alpha a(x^{2})+\beta a(x)^{2} 
 \qquad 
 \left(x\in \mathbb{F}\right), 
\]
where the additive function $a\colon \mathbb{F}\to \mathbb{C}$ is defined by 
\[
 a(x)= F_{2}(x, 1) 
 \qquad 
 \left(x\in \mathbb{F}\right). 
\]
Writing this form back into equation \eqref{Eq2}, we deduce 
\[
 \alpha a(x^{2n})+\beta a(x^{n})^{2}= \left(\alpha a(x^{2})+\beta a(x)^{2}\right)^{n} 
 \qquad 
 \left(x\in \mathbb{F}\right). 
\]
Again, both the sides of this identity are traces of symmetric and $2n$-additive functions, therefore 
\begin{multline*}
 \frac{1}{(2n)!} \sum_{\sigma\in \mathscr{S}_{2n}} \left[\alpha a(x_{\sigma(1)} \cdots x_{\sigma_{(2n)}})+\beta a(x_{\sigma_{(1)}} \cdots x_{\sigma_{(n)}})a(x_{\sigma(n+1)} \cdots x_{\sigma(n)})\right]
\\= 
\frac{1}{(2n)!} \sum_{\sigma\in \mathscr{S}_{2n}} \prod_{k=0}^{n-1}(\alpha a(x_{\sigma(2k+1)}x_{\sigma(2k+2)})+\beta a(x_{\sigma(2k+1)})a(x_{\sigma(2k+2)}))
\end{multline*}
for each $x_{1}, \ldots, x_{2n}\in \mathbb{F}$. Let now $x, y, z\in \mathbb{F}$ be arbitrary, then this identity with the substitutions 
\[
 x_{1}=x, \, x_{2}=y,\, x_{3}=z \quad  \text{and} \quad x_{i}=0 \quad \text{for} \quad 4\leq i \leq 2n
\]
leads to 
\[
 Aa(xyz)+Ba(xy)a(z)+Ca(xz)a(y)
 +Da(yz)a(x)+Ea(x)a(y)a(z)=0 
 \qquad 
 \left(x, y, z\in \mathbb{F}\right), 
\]
that is a similar equation that appear in the proof of Theorem \ref{THM_main}. With an analogous thread we get that $a$ can be written as a sum of two homomorphisms that finally implies for the function $f$ that there are homomorphisms $\varphi_{1}, \varphi_{2}\colon \mathbb{F}\to \mathbb{C}$ such that 
\[
 f(x)= f(1)\varphi_{1}(x)\varphi_{2}(x) 
 \qquad 
 \left(x\in \mathbb{F}\right), 
\]
where $f(1)$ is either zero, or it is an $(n-1)$\textsuperscript{st} root of unity. 
\end{proof}

\section{Open problems and further perspectives}

As it is written at the beginning of the second section, the main aim of this paper was to 
prove characterization theorems for generalized polynomials $f\colon \mathbb{F}\to \mathbb{C}$ of degree at most $n$ that also fulfill equation 
\[
 f(P(x))= Q(f(x))
\]
for each $x\in \mathbb{F}$, 
where $n$ is a positive integer and $P\in \mathbb{F}[x]$ and $Q\in \mathbb{C}[x]$ are polynomials. 
The results presented in connection with this problem can be considered as initial steps. Thus we close this paper with several open questions and we would like to give some perspectives, too. 

\begin{rem}
 Clearly, it is enough to consider the case $\deg (P)=\deg(Q)$. Indeed, if 
 $f\colon \mathbb{F}\to \mathbb{C}$ is a generalized monomial of degree $n$, the due to Lemma \ref{lemma_6}, the mappings 
 \[
  \mathbb{F}\ni x \longmapsto f(P(x)) 
  \quad 
  \text{and}
  \quad 
  \mathbb{F}\ni x \longmapsto Q(f(x)) 
 \]
are generalized polynomials of degree $n\cdot \deg(P)$ and $n\cdot \deg(Q)$, respectively. Furthermore, 
they coincide at each point $x\in \mathbb{F}$. However, this is only possible if $\deg (P)=\deg(Q)$. 
\end{rem}

In the second section, we studied only quadratic functions, i.e., generalized monomials of degree two. Therefore, we formulate the following. 

\begin{opp}
Let  $n\in \mathbb{N}$, $n\geq 2$ and $P\in \mathbb{F}[x]$ and $Q\in \mathbb{C}[x]$ be polynomials of degree at least two and $f\colon \mathbb{F}\to \mathbb{C}$ be a generalized monomial of degree $n$. Prove or disprove that if 
\[
 f(P(x))= Q(f(x)) 
 \qquad 
 \left(x\in \mathbb{F}\right), 
\]
then and only then there exist homomorphisms $\varphi_{1}, \ldots, \varphi_{n}\colon \mathbb{F}\to \mathbb{C}$ such that 
\[
 f(x)= f(1)\cdot \varphi_{1}(x)\cdots \varphi_{n}(x)
 \qquad 
 \left(x\in \mathbb{F}\right). 
\]

\end{opp}

\begin{opp}
 It might be promising to consider firstly the case 
 \[
  P(x)= Q(x)=x^{2} 
  \qquad 
  \left(x\in \mathbb{F}\right), 
 \]
 because as the results of the previous section show, hopefully, the case 
 \[
  P(x)= Q(x)=x^{n} 
  \qquad 
  \left(x\in \mathbb{F}\right)
 \]
 where $n>2$, leads back to the case $n=2$. 
\end{opp}

\begin{opp}
 In the special case we considered the above general problem, it turned our that the solutions of the functional equations are always (normal) monomials. Additionally, we also showed that the proof is much easier if we know this, see Proposition \ref{Prop1}. We conjecture that this is true in general, too. Thus, prove or disprove that if the generalized monomial $f\colon \mathbb{F}\to \mathbb{C}$ solves equation 
 \[
 f(P(x))= Q(f(x)) 
 \qquad 
 \left(x\in \mathbb{F}\right), 
\]
then $f$ is a monomial. 
\end{opp}

\begin{rem}
 The above problem is clearly meaningful for polynomials $P$ and $Q$ with degree one. At the same time, in this case we cannot expect representations similar to that ones that appeared in the statements proved in the second section. Here we consider only the case quadratic functions. Accordingly, assume that for the quadratic function $f\colon \mathbb{F}\to \mathbb{C}$ equation 
 \[
  f(ax+b)= Af(x)+B 
  \qquad 
  \left(x\in \mathbb{F}\right), 
 \]
with some $a, b\in \mathbb{F}$ and $A, B\in \mathbb{C}$. Let us denote $F\colon \mathbb{F}^{2}\to \mathbb{C}$ the uniquely determined symmetric, bi-additive function whose trace is $f$. With $x=0$ we immediately get that $f(b)=F(b, b)=B$. 
Furthermore, 
\[
 F(ax+b, ax+b)=AF(x, x)+F(b, b)
 \qquad 
 \left(x, \in \mathbb{F}\right), 
\]
that is, 
\[
 F(ax, ax)+2F(ax, b)+F(b, b)= AF(x, x)+F(b, b)
 \qquad 
 \left(x, \in \mathbb{F}\right), 
\]
or after some simplification 
\[
 \left(F(ax, ax)-AF(x, x)\right)+2F(ax, b)= 0
\]
for all $x\in \mathbb{F}$. Since the left hand side of this equation is a generalized polynomial of degree two which has to be identically zero, all of its monomial terms should vanish. 
This means from one hand that 
\[
 F(ax, b)=0
 \qquad 
 \left(x\in \mathbb{F}\right), 
\]
especially, $f(b, b)=B=0$. On the other hand, we also have 
\[
 F(ax, ax)=AF(x, x)
 \qquad 
 \left(x\in \mathbb{F}\right), 
\]
from this however 
\[
 F(ax, ay)=AF(x, y)
 \qquad 
 \left(x, y\in \mathbb{F}\right), 
\]
follows, that is, the symmetric, bi-additive function is \emph{semi-homogeneous}. 
From \cite[Theorem 3]{GseKisVin20} it is known that a non-identically zero, symmetric and bi-additive function 
$F$ fulfilling the above semi-homogeneity exists if and only if there are injective homomorphism 
$\delta_{1}, \delta_{2}\colon \mathbb{F}\to \mathbb{C}$ such that 
\[
 \delta_{1}(a)\delta_{2}(a)= A. 
\]
Summing up, if $f\colon \mathbb{F}\to \mathbb{C}$ is a non-identically zero quadratic function such that 
\[
  f(ax+b)= Af(x)+B 
  \qquad 
  \left(x\in \mathbb{F}\right), 
 \]
with some $a, b\in \mathbb{F}$ and $A, B\in \mathbb{C}$, then 
\begin{enumerate}[(i)]
 \item $B=0$
 \item for the uniquely determined symmetric and bi-additive function $F\colon \mathbb{F}^{2}\to \mathbb{C}$, identity
 \[
  F(ax, b)=0
 \]
is satisfied for all $x\in \mathbb{F}$. 
\item there are injective homomorphism 
$\delta_{1}, \delta_{2}\colon \mathbb{F}\to \mathbb{C}$ such that 
\[
 \delta_{1}(a)\delta_{2}(a)= A. 
\]
\end{enumerate}

Observe that for instance with $b=0$, with arbitrary fixed $a\in \mathbb{Q}$ and with 
$A=a^{2}$ the above identity is fulfilled by \emph{any} quadratic function $f\colon \mathbb{F}\to \mathbb{C}$ (this is obviously consistent with the fact that quadratic functions are $\mathbb{Q}$-homogeneous of degree two). This shows that in this case we do not get in general any information for the form of the involved quadratic function $f$. With an analogous method, we obtain the same for higher order generalized monomials. 
\end{rem}

\begin{opp}
 In this paper only equations with one unknown function were considered. At the same time, the investigated problem can clearly be extended: let 
 $P\in \mathbb{F}[x]$ and $Q\in \mathbb{C}[x]$ be polynomials, $f, g\colon \mathbb{F}\to \mathbb{C}$ be generalized polynomials (of possibly \emph{different order}) 
 such that $\deg(f)\deg(P)=\deg(g)\deg(Q)$. Prove or disprove that if 
 \[
  f(P(x))= Q(g(x)) 
 \]
holds for all $x\in \mathbb{F}$, then $f$ and $g$ can be represented as products of homomorphisms. 
\end{opp}

\begin{rem}
 A particularly interesting and presumably the simplest case of the following problem is when $\deg(g)=1$, i.e., when $g\colon \mathbb{F}\to \mathbb{C}$ is an additive function. 
 \end{rem}
 
 In connection to this problem we prove the following special case, which is expected to be successfully applied in the general case as well. 
 
 \begin{prop}
  Let $\mathbb{F}$ be a field with $\mathrm{char}(\mathbb{F})=0$, $f\colon \mathbb{F}\to \mathbb{C}$ be a quadratic function and $a\colon \mathbb{F}\to \mathbb{C}$ be an additive function. Then equation 
  \[
  f(x^{2})= a(x)^{4} 
  \qquad 
  \left(x\in \mathbb{F}\right)
 \]
 holds if and only if here exists a homomorphism $\varphi\colon \mathbb{F}\to \mathbb{C}$ such that 
\[
 a(x)=a(1)\varphi(x) 
 \qquad 
 \left(x\in \mathbb{F}\right)
\]
and 
\[
 f(x)= a(1)^{4}\varphi(x)^{2}
 \qquad 
 \left(x\in \mathbb{F}\right). 
\]
\end{prop}

\begin{proof} 
 Assume that $f\colon \mathbb{F}\to \mathbb{C}$ is a quadratic, while $a\colon \mathbb{F}\to\mathbb{C}$ is an additive function such that 
 \[
  f(x^{2})= a(x)^{4} 
  \qquad 
  \left(x\in \mathbb{F}\right). 
 \]
Since both the sides of this equation are traces of symmetric, $4$-additive functions, we obtain that 
\[
 \frac{1}{3}\left[F(x_1 x_2, x_3 x_4)+F(x_1 x_3, x_2 x_4)+F(x_1 x_4, x_2 x_3)\right]
 = a(x_1)a(x_2)a(x_3)a(x_4) 
 \qquad
 \left(x_1, x_2, x_3, x_4\in \mathbb{F}\right). 
\]
Here $F\colon \mathbb{F}\times \mathbb{F}\to \mathbb{C}$ is the uniquely determined symmetric, bi-additive function for which $F(x, x)= f(x)$ is satisfied for all $x\in \mathbb{F}$. 

This identity implies especially that 
\[
 f(1)=F(1, 1)=a(1)^{4} 
 \qquad 
 \text{and}
 \qquad 
 F(x, 1)=a(1)^{3}a(x) 
 \qquad 
 \left(x\in \mathbb{F}\right). 
\]
Furthermore, we also have 
\[
 2F(x, y)= 3a(1)^{2}a(x)a(y)-F(xy, 1) 
 \qquad 
 \left(x, y\in \mathbb{F}\right). 
\]
Thus, 
\[
 2F(x, y)=3a(1)^{2}a(x)a(y)-a(1)^{3}a(xy) 
 \qquad 
 \left(x\in \mathbb{F}\right). 
\]
In other words, 
\[
 f(x)= \frac{3}{2}a(1)^{2}a(x)^{2}-\frac{1}{2}a(1)^{3}a(x^{2})  
 \qquad 
 \left(x\in \mathbb{F}\right). 
\]
Substituting this back into the original equation we get that the additive function $a\colon \mathbb{F}\to \mathbb{C}$ has to fulfill 
\[
 a(x)^{4}= \frac{3}{2}a(1)^{2}a(x^{2})^{2}-\frac{1}{2}a(1)^{3}a(x^{4})  
 \qquad 
 \left(x\in \mathbb{F}\right). 
\]
Again, after symmetrization or after applying \cite[Theorem 15]{GseKisVin19}, we derive that there exists a homomorphism 
$\varphi\colon \mathbb{F}\to \mathbb{C}$ such that 
\[
 a(x)=a(1)\varphi(x) 
 \qquad 
 \left(x\in \mathbb{F}\right)
\]
and 
\begin{multline*}
 f(x)= \frac{3}{2}a(1)^{2}a(x)^{2}-\frac{1}{2}a(1)^{3}a(x^{2})
 \\
 = 
 \frac{3}{2}a(1)^{2}a(1)^{2}\varphi(x)^{2}-\frac{1}{2}a(1)^{3}a(1)\varphi(x^{2})
 =
 a(1)^{4}\varphi(x)^{2} 
 \qquad 
 \left(x\in \mathbb{F}\right). 
\end{multline*}
\end{proof}

\begin{rem}
Using the ideas of the results of the third section, the case 
\[
 f(x^{n})= a(x)^{2n} 
 \qquad 
 \left(x\in \mathbb{F}\right)
\]
where $n$ is a fixed positive integer, can be reduced to the above studied case. 
\end{rem}

\begin{rem}
 Observe that during the proof of Lemmas \ref{lemma3}, \ref{lemma5}, \ref{lemma_6} the fact that we considered complex-valued mappings, was not used at all. We remark that these statements as well as their proofs are exactly the same for functions defined on a field $\mathbb{F}$ and mapping to another field $\mathbb{K}$. 
 
 The situation is slightly different in case of Lemma \ref{lemma4}, since in that case derivations are involved. Nevertheless, Lemma \ref{lemma4} also holds true (with an unchanged proof) for mappings defined on $\mathbb{F}$ and having values in $\mathbb{K}$, where $\mathbb{F}\subset \mathbb{K}$ are fields. 
\end{rem}

\begin{rem}
 We would also like to clarify why we considered only complex-valued functions. Obviously, the investigated problems are meaningful in a much more general setting. The importance of this condition lies in our method. Namely, in each case we showed firstly that the involved mappings $f\colon \mathbb{F}\to \mathbb{C}$ are exponential polynomials of the multiplicative group $\mathbb{F}^{\times}$. This enabled us to describe the unknown function completely. This method however relies on the notion of exponential polynomials and the theory of this notion is well-developed only for complex-valued mappings. Maybe with a different approach the general case can also be handled. 
 
 The assumption that the field $\mathbb{F}$ has to have zero characteristic is caused by the limitations of the Polarization formula, since $n$-additive functions are uniquely determined by their diagonalizations only if the characteristic of the domain is large enough or zero. 
\end{rem}

\begin{opp}
 Let $\mathbb{F}$ and $\mathbb{K}$ be fields, $n$ be a positive integer, $P\in \mathbb{F}[x]$ and $Q\in \mathbb{K}[x]$ be polynomials. Determine those generalized monomials $f, g\colon \mathbb{F}\to \mathbb{K}$ of degree at most $n$ that also fulfill equation 
\[
 f(P(x))= Q(f(x))
\]
for each $x\in \mathbb{F}$. 

A particularly interesting case of this problem is when at least one of the fields $\mathbb{F}$ and $\mathbb{K}$ has finite characteristic. 
\end{opp}

\begin{ackn}
 The author would like to thank Professor \textsc{José María Almira} (University of Murcia, Spain) and Professor \textsc{Rezső Lovas} (University of Debrecen, Hungary) for their generous help and valuable comments that improved the quality of the paper. 
 
 Project No.~K134191 has been implemented with the support provided by the National Research, Development and Innovation Fund of Hungary, financed under the K\_20 funding scheme. The research of the author has partially been carried out with the help of the Project 2019-2.1.11-TÉT-2019-00049, which has been implemented with the support provided from the National Research, Development and Innovation Fund of Hungary, financed under the TÉT funding scheme.
\end{ackn}

 \bibliographystyle{plain}

 \end{document}